\pdfoutput=1
\documentclass[a4paper]{amsart}
\usepackage{amssymb}
\usepackage[T1]{fontenc}
\usepackage{lmodern}
\usepackage{mathrsfs}
\usepackage[stretch=10, shrink=10]{microtype}
\usepackage{xcolor}
\usepackage{enumerate}
\usepackage{graphicx}
\usepackage[backend=bibtex8]{biblatex}
\usepackage[colorlinks,
linkcolor={red!50!black}, citecolor={green!40!black}, urlcolor={blue!50!black}
]{hyperref}

\newtheorem{thm}{Theorem}[section]
\newtheorem{cor}[thm]{Corollary}
\newtheorem{lem}[thm]{Lemma}

\newtheorem*{prop*}{Proposition}
\newtheorem*{fact*}{Fact}
\newtheorem*{claim*}{Claim}

\theoremstyle{definition}

\newtheorem{rem}{Remark}

\title[Fra\"iss\'e
  limit of finite Heyting algebras]{The small index property of the Fra\"iss\'e
  limit of finite Heyting algebras}
\author{Kentar\^o Yamamoto}
\address{Institute of Computer Science of the Czech Academy of Sciences\\
  Pod Vod\'arenskou v\v{e}\v{z}\'i  271/2\\
  Libe\v{n}\\
  182 07 Prague\\ The Czech Republic}
\newcommand{\continuum}{2^{\aleph_0}}
\newcommand{\inv}{{-1}}
\newcommand{\FR}{Fra\"iss\'e}
\newcommand{\tp}{\mathrm{tp}}
\newcommand{\sunim}{\mathbin{\rightarrow}}

\newcommand{\0}{{\mathclap{\phantom{\emptyset}}0}}
\IfFileExists{emptysetiszero}{}{\renewcommand{\0}{\emptyset}}
\newcommand{\power}{\mathscr{P}}

\newcommand{\Aut}{\mathrm{Aut}}

\newcommand{\indep}[1][]{%
  \mathrel{
    \mathop{
      \vcenter{
        \hbox{\oalign{\noalign{\kern-.3ex}\hfil$\vert$\hfil\cr
            \noalign{\kern-.7ex}
            $\smile$\cr\noalign{\kern-.3ex}}}
      }
    }\displaylimits_{#1}
  }
}
\newcommand{\gen}[1]{\langle#1\rangle}
\newcommand{\dcl}{{\mathrm{dcl}}}

\DeclareMathOperator{\dom}{dom}

\DeclareMathOperator{\ran}{ran}
\DeclareMathOperator{\supp}{supp}

\begin{document}
\hyphenation{super-amalga-mation}
\hyphenation{auto-homeo-morphisms}
\hyphenation{ultra-homo-gene-ous}
\hyphenation{non-empty}
%\pretolerance=-1
\begin{abstract}
  We show that if a subgroup of the automorphism group of the \FR{} limit of
  finite Heyting algebras has a countable index,
  then it lies between the pointwise and setwise stabilizer of some finite set. 
\end{abstract}
\maketitle
\section{Introduction}
The study of automorphism groups of countable ultrahomogeneous structures
has been carried out in the field for some time.
Here, a structure is ultrahomogeneous,
or homogeneous in the sense of \FR{}~\cite{fraisse53:_sur_certain_relat_gener_nombr_ration},
if every isomorphism between finitely generated substructures
can be extended to an automorphism of the structure.
One theme that frequents in this line of research is
the \emph{strong small index property} of an automorphism group:
$\Aut(M)$ for some structure $M$ has the strong small index property
if every subgroup of countable index sits
between the pointwise stabilizer of some finite set in $M$
and the setwise stabilizer of the same set.
(We simply say that $M$ has the said property in that case.)
Since every open nontrivial subgroup of $\Aut(M)$ as a topological group
with the pointwise convergence topology
has a countable index if $M$ is countable,
the topology of $\Aut(M)$ with the strong small index property
can be recovered from its abstract group structure alone.
Dixon, Neumann, and Thomas~\cite{dixon86:_subgr_small_index_infin_symmet_group} showed the strong small index property
of the pure set~$\aleph_0$,
and Truss~\cite{TRUSS1989494} established it for the countable atomless Boolean algebra.
Another work more related to ours is by Droste and Macpherson~\cite{droste00:_autom_group_univer_distr_lattic},
who showed the strong small index property
for the universal ultrahomogeneous distributive lattice.
The goal of this article is to show an analogous statement
for Heyting algebras.

The central object of study of the present article is
the \FR{} limit $L$ of finite Heyting algebras.
It is the unique (up to isomorphism) countable ultrahomogeneous structure
into which all nontrivial finite Heyting algebras
embed.
Its existence and uniqueness follows from a general fact~\cite{fraisse53:_sur_certain_relat_gener_nombr_ration} in model theory
and the amalgamation property of the class of finite Heyting algebras
(in fact, a stronger property called the \emph{superamalgamation property}~\cite{Maksimova1977} holds in the class).
Recall that Heyting algebras are bounded distributive lattices $H$ in which
the minimum element (written $x \sunim y$)
of $\{z \in H \mid x \wedge z \le y\}$ exists
for every $x, y \in H$;
they are for our purposes structures in the language $\sigma_{\mathrm{HA}} = \{0, 1, \wedge, \vee, \sunim\}$.
The interest in studying the \FR{} limit $L$ of finite Heyting algebras
and its automorphism group is twofold.
First,
since there are unboundedly large $1$-generated finite Heyting algebra,
$L$ is not $\omega$-categorical
unlike aforementioned better-known ultrahomogeneous structures.
Hence, we provide a novel case study in the context of research in
automorphism groups of ultrahomogeneous structures.
Secondly, $L$ can be thought of as a midway between the said examples
and the universal ultrahomogeneous poset,
whose small index property is usually believed but has not yet been proved.
In fact,
$\Aut(L)$ can be thought of the automorphism group of the  ``random poset''
in a different sense:
there is a structure with the same automorphism group as $L$,
which is called the Esakia space of $L$ in \S~\ref{sec:preliminaries},
whose quotients exhaust all finite posets
(see also Remark~\ref{rem:projective}). 

Our main result is Theorem~\ref{thm:ssip},
in which we prove that $L$ has the strong small index property.
The proof is based on Truss's method~\cite{TRUSS1989494}.
After that we also prove a model-theoretic consequence of the Theorem
regarding a weak form of elimination of imaginaries (Corollary~\ref{cor:wei}).

In the final section we discuss newer and more systematic methodology
in establishing the strong amalgamation property
and the apparent difficulty in applying it to the present problem.

\section{Preliminaries}\label{sec:preliminaries}
A \emph{lower amalgam} (or an amalgam to some authors, e.g., \cite{Maksimova1977}) is
a pair of embeddings from the same structure.
A lower amalgam $A_1 \stackrel{\iota_1}{\hookleftarrow} A_0 \stackrel{\iota_2}{\hookrightarrow} A_2$
is \emph{completed} by a structure $A$
if there are embeddings $e_i : A_i \hookrightarrow A$ ($i < 2$)
such that $e_1 \circ \iota_1 = e_2 \circ \iota_2$.
A lower amalgam is in the \emph{normal form}  if $A_0 = A_1 \cap A_2$
and both $\iota_1, \iota_2$ are inclusion maps.
A class of structures $\mathcal K$ has the \emph{amalgamation property}
if any lower amalgam consisting of embeddings in $\mathcal K$
can be completed by a structure in $\mathcal K$.
If $\mathcal K$ is closed under isomorphisms,
which we will assume tacitly hereafter,
then $\mathcal K$ has the amalgamation property
whenever each lower amalgams of the normal form in $\mathcal K$ can be completed there.
Therefore, unless otherwise stated,
we mean lower amalgam of the normal form by ``lower amalgams'' \emph{simpliciter}.

It now makes sense to discuss stronger forms of this property, \emph{viz}.,
the \emph{strong amalgamation property} and the \emph{superamalgamation property}.
A class of structures $\mathcal K$ has the strong amalgamation property
if in the definition of completion of amalgams above,
$e_1$ and $e_2$ can be chosen so that $e_1(\iota_1(A_0)) = e_2(\iota_2(A_0))$.
If $\mathcal K$ has the strong amalgamation property,
then
$e_1$ and $e_2$ may be taken to be inclusion maps without loss of generality.
The class $\mathcal K$ has the superamalgamation property
if every lower amalgam $A_1 \stackrel{\iota_1}{\hookleftarrow} A_0 \stackrel{\iota_2}{\hookrightarrow} A_2$ in $\mathcal K$ can be completed
by $A \in \mathcal K$ and inclusion maps so that $A_1 \indep'_{A_0} A_2$,
which is defined to be equivalent to
\[
  \bigvee_{(i, j) = (1, 2), (2, 1)}
  (\forall a_i \in A_i) (\forall a_j \in A_j) [a_i \le a_j \implies
  \mathop{(\exists a_0 \in A_0)} a_i \le a_0 \le a_j],
\]
where $\le$ is the partial order defined on $A$,
which induces the same partial orders on $A_i$ ($i = 0, 1, 2$).

Suppose that a class $\mathcal K$ of finitely generated structures
has the amalgamation property in general,
is closed under substructures,
has the joint embedding property
(i.e., every two structures in $\mathcal K$ have a common extension in $\mathcal K$),
and has at most countably many structures up to isomorphism.
That is, $\mathcal K$ is a \emph{\FR{} class}.
Then by a theorem of \FR~\cite{fraisse53:_sur_certain_relat_gener_nombr_ration},
there is a unique countable structure $L$, the \emph{\FR{} limit} of $\mathcal K$,
such that the class of finitely generated substructures, or the \emph{age},
of $L$ is exactly $\mathcal K$,
and that $L$ is \emph{ultrahomogeneous}:
every isomorphism between finitely generated substructures of $L$
can be extended to an automorphism of $L$.

Maksimova~\cite{Maksimova1977} studied the three variants of the amalgamation property
for classes of Heyting algebras
and showed, among other things,
that the class of all Heyting algebras has the superamalgamation property.
Her argument actually shows that the class $\mathcal K$ of \emph{finite} Heyting algebras
has the superamalgamation property
(and thus the strong amalgamation property),
from which it easily follows that $\mathcal K$ is a \FR{} class.
Let $L$ be its \FR{} limit.
Let $X$ be the Esakia space 
of $L$ (\cite{Esakia2019}).
Concretely,
$X$ is the set of prime filters of $L$,
equipped with the partial order of set inclusion
and
the topology generated by sets of the form $\widehat a$
or $X \setminus \widehat a$ for $a \in L$,
where $\widehat a = \{x \in X \mid a \in x\}$.
This is the \emph{patch topology} on $X$.
One can show \cite[Lemma 2.3.2]{yamamotoar:_autom_fra_limit_heytin_algeb} that \emph{the Boolean envelope} of $B(L)$,
or the Stone dual of the topological space $X$,
is simply the countable atomless Boolean algebra.
Hence, as a topological space $X$ is the Cantor space.
There is another, coarser topology on $X$
that is generated by sets of the form $\widehat a$,
which will then be exactly the compact open subsets
(see, e.g., Cornish~\cite{cornish75:_h}).
This is the \emph{spectral topology} on $X$.
Unless otherwise stated, we equip $X$ with the patch topology.

Consider ${}^\circ : B(L) \to B(L)$ that maps a clopen subset of $X$
to the maximum up-set contained in $X$, which is clopen.
This map is intensive, idempotent, and preserves possibly empty finite meets,
i.e., $(B(L), {}^\circ)$ is an \emph{interior algebra} \cite{blok76:_variet_inter_algeb},
which is a structure in the language $\sigma_{\mathrm{int}} = \{0, 1, \wedge, \vee, \neg, {}^\circ\}$.
Moreover, $(B(L), {}^\circ)$
is \emph{skeletal}~\cite{Esakia2019} in that $B(L)$ is generated as a Boolean algebra
by $B(L)^\circ = L$.
One can easily see that the age of $(B(L), {}^\circ)$ has
the superamalgamation property.
Let $G$ be $\Aut(L) = \Aut(X)$.
In \cite{yamamotoar:_autom_fra_limit_heytin_algeb}, the superamalgamation property of $L$
and a resulting stationary independent relation $\indep$ \cite{CALDERONI202143}
were used to prove the simplicity of $G$.
Here, $\indep$ is a ternary relation among finite subsets
(not necessarily subalgebras) of $L$ defined so that 
$A_1 \indep_{A_0} A_2$ is equivalent to
$\langle A_1 A_0 \rangle^L \indep'_{\langle A_0\rangle^L} \langle A_0A_2 \rangle^L$.

The argument used there
is also applicable to the automorphism group of an arbitrary
countable ultrahomogeneous Heyting algebra whose age has
the superamalgamation property
and hence to that of an arbitrary countable ultrahomogeneous
skeletal interior algebra with the same condition
due to the isomorphism of the category of Heyting algebras
and that of skeletal interior algebras.

We will not go into details of the Esakia duality.
We, however,
note the following facts that can easily be seen from the definitions
(see, e.g., \cite{Esakia2019}):
\begin{enumerate}
\item
  If $X_1 \subseteq X_2$ are the Esakia spaces of $L_1$ and $L_2$,
  respectively, and
  $X_1$ is a clopen up-set in $X_2$,
  then the inclusion map $X_1 \hookrightarrow X_2$ is an \emph{Esakia space morphism},
  i.e., a function induced by some Heyting algebra homomorphism $L_2 \to L_1$.
\item \sloppy
  The automorphisms on $L$ induce exactly the
  autohomeomorphisms (with either of the two topologies) that are
  automorphisms of the partial order reduct. \fussy
\end{enumerate}
\begin{rem}\label{rem:projective}
  The sense in which $\Aut(L)$ is the automorphism group of the ``random poset''
  of some sort is as follows.
  The Esakia duality restricts to the duality between finite Heyting algebras
  and finite Esakia spaces,
  which are exactly the finite posets (with discrete topology).
  In projective \FR{} theory~\cite{irwin06:_projec_fraïs_limit_pseud_arc}
  that replaces embeddings in the classical \FR{} theory
  by surjective continuous maps,
  it makes sense to talk about the projective \FR{} limit $\mathbb P$
  of the finite Esakia spaces.
  One observes that
  the group of autohomeomorphisms of $\mathbb P$ that are order-automorphisms
  is topologically isomorphic to $\Aut(L)$.
\end{rem}
Whenever
$Y \subseteq X$, we write $\Aut_X(Y)$ for
$\{g \in \Aut(X) \mid \supp g \subseteq Y\}$.
Given a family $F$ of automorphisms of $X$,
an \emph{agglutinate}~$\sigma: X \to X$ of $F$ is an automorphism of $X$
such that $\sigma(x) = \sigma'(x)$ where $\sigma'$ is the (unique) automorphism in $F$
whose support contains $x$ if such  $\sigma'$ exists,
and that $\sigma(x) = x$ otherwise.

\section{Small index property}

\begin{lem}\label{lem:homog}
  Let $Y \subseteq X$ be clopen and nonempty.
  The set $Y$ equipped with the induced topology and order
  is an Esakia space dual to a skeletal  interior algebra
  whose age has the superamalgamation property.
  \emph{A fortiori},
  $\Aut_X(Y)$ is simple.
\end{lem}
\begin{proof}
  Since $Y$ is clopen, it is an Esakia space \cite[]{Esakia2019},
  and it is the dual of some countable skeletal interior algebra $A$.
  We may identify $A$ as the induced partial order on $\mathop{\downarrow} a$
  for some $a \in B(L)$, the Boolean envelope of $L$.
%  (The Heyting operation of $A$ does not in general coincide with the restriction of that of $L$.)

  Consider the translation ${}^*$
  from $\sigma_{\mathrm{int}}$ to $\sigma_{\mathrm{HA}}$
  replacing every occurrence of the constant~$1$ with $a$ and
  every occurrence of the function symbol $\neg(\cdot)$ by
  $\neg(\cdot) \wedge a$.
  By construction, we have
  \[
    A \models \phi \iff L \models \phi^*
  \]
  for every quantifier-free sentence $\phi$ in $\sigma_{\mathrm{int}}(A)$.
  
  We prove the superamalgamation property for the age of $A$.
  Suppose that a lower amalgam
  $A_1 \hookleftarrow A_0 \hookrightarrow A_2$ of $A$
  is given.
  For ease of notation, we assume that the embeddings in the diagram
  are inclusions.
  Fix an enmueration of elements of $A_2 \setminus A_0$,
  and let $\Phi(\bar x)$
  be the set of formulas over $A_0$
  such that $\Phi(A_2 \setminus A_0)$ is the diagram of $A_2$
  in $\sigma_{\mathrm{int}}$.
  Then $\Phi^*(\bar x)$ is the set of formulas over $aA_0$
  such that $L \models \Phi^*(A_2 \setminus A_0)$.
  By the Existence Axiom of $\indep$,
  there is $A_2' \subseteq L$ such that $L \models \Phi^*(A_2' \setminus A_0)$
  and $A_1 \indep[aA_0] A_2'$.
  By construction, $A_2'$ is below $a$,
  and $\langle A_1A_2' \rangle^A$ completes the lower amalgam.
\end{proof}

\begin{lem}\label{lem:partition}
  For each nonempty clopen upward closed $Y \subseteq X$,
  there exist nonempty clopen up-sets $U, V \subseteq Y$
  such that $U \cap V = \0$ and $U \cup V = Y$.
\end{lem}
\begin{proof}
   Since $L$ is an existentially closed Heyting algebra,
   one can take \cite[Proposition A.2.(ii)]{GHILARDI199727}
   a nonempty clopen up-set $U \subsetneq Y$
  which as an element of $L$ has a $L$-complement $U'$.
  Let $V = U' \cap Y$.
\end{proof}

\begin{lem}\label{lem:trans}
  Let $Y \subseteq X$ be nonempty and finite.
  Then the $G$-orbit $G \cdot Y$
  of $Y$ in the natural action $G \curvearrowright \power(X)$
  is of cardinality continuum.
\end{lem}
\begin{proof}
  For ease of notation, assume $|Y| = 1$
  (this argument is general, however).
  Let $y$ the prime filters in $L$ such that $Y = \{y\}$.
  We will construct $\{b^j_i\}_{i<\omega} \subseteq y$ ($j = 0, 1$),
  finite isomorphisms $\sigma^\rho$, $e_\rightarrow^\rho$, $e_\leftarrow^\rho$
  ($i < \omega$, $\rho \in  2^i$),
  $\{d_i\}_{i < \omega}, \{r_i\}_{i<\omega} \subseteq L$,
  $\{A_i\}_{i< \omega} \subseteq \mathrm{Age}(L)$
  by simultaneous induction.
  % The family $\{A_i\}$ will be a chain, and
  % for $\rho'$ extending $\rho$
  % the isomorphism $\sigma^{\rho'}$ will extend $\sigma^\rho$.
  Given all objects are defined for the index $i-1$,
  first take $b^0_i \in y \setminus A_{i-1}$,
  which exists since $A_{i-1}$ is finite.
  Now take $b^1_i \indep[A_{i-1}]b^0_i$.
  Let $\tau^j$ ($j < 2$) be the automorphism of $\gen{A_{i-1}b_i^0b_i^1}$
  such that $\tau_i^j \upharpoonright A$ is the identity
  and $\tau_i^j(b_i^0b_i^1) = b_i^{j_0}b_i^{j_1}$,
  where $(j_0, j_1) = (0, 1)$ if $j=0$ and $(j_0, j_1) = (1, 0)$ otherwise.
  Now let $\sigma^{\rho j} = \tau_i^j \circ e^\rho_\leftarrow$.
  Take $d_i \in L \setminus \gen{A_{i-1}b_i^0b_i^1}$.
  Apply the weak homogeneity \cite[\S~7.1]{hodges_1993}
  to the embedding $\iota \circ (\sigma^{\rho j})^\inv$,
  where $\iota$ is the inclusion map $\gen{A_{i-1}b_i^0b_i^1d_i} \hookrightarrow$,
  to obtain the isomorphism
  $e^{\rho j}_\rightarrow$ of $\gen{A_{i-1}b_i^0b_i^1d_i}$ into an age of $L$.
  Now take
  $r_i \in L \setminus
  \gen{A_{i-1}b_i^0b_i^1d_ie_\rightarrow^{\rho 0}(d_i)e_\rightarrow^{\rho 1}(d_i)}$.
  Use the similar argument to obtain
  the isomorphism
  $e_\leftarrow^{\rho j} : \gen{A_{i-1}b_i^0b_i^1d_ie_\rightarrow^{\rho 0}(d_i)e_\rightarrow^{\rho 1}(d_i) r_i}$ into an age of $L$.
  Finally, let $A_i = \gen{A_{i-1}b_i^jd_ie_\rightarrow^{\rho j}(d_i)e_\leftarrow^{\rho j}(r_i)}_{j < 2}$.

  By construction, $\{A_i\}_{i<\omega}$ is a chain
  such that $L = \bigcup_{i<\omega} A_i$.
  Moreover,
  for $\rho'$ extending $\rho$
  the isomorphism $\sigma^{\rho'}$ extends $\sigma^\rho$,
  and for $T \in 2^\omega$, $\sigma^T := \bigcup_{i < \omega} \sigma^{T \upharpoonright n}$
  is an automorphism of $L$.
  By construction, $\sigma^T\cdot y$ are all distinct
  and thus $|G \cdot y| = \continuum$.

  The general case can be made by having $n$ arrays
  corresponding to $\{b^j_i\}$ if $|Y| = n$
  essentially because a $n$-branching tree of a finite depth is finite.
\end{proof}

Recall that $\Aut_X(Y) = \{\sigma \in G \mid \operatorname{supp} \sigma \subseteq Y\} = \{ \sigma \in G \mid \sigma \upharpoonright (X \setminus Y) = \mathrm{id}_{X \setminus Y}\}$.

\begin{lem} \label{lem:adam}
  Let $Y \subseteq X$ be nonempty, closed,
  and nowhere dense in the spectral topology,
  and $U \subseteq X$ a basic open set intersecting $Y$ nontrivially.
  Then there is $\sigma \in \Aut_X(U)$ that does not fix $Y$ setwise.
\end{lem}
\begin{proof}
  Throughout the proof, we equip $X$ with the spectral topology.
  Note that since $U$ is compact, so is $U \cap Y$.
  There is a compact open subset $V$ of $U$ containing $U \cap Y$.
  Indeed, for each $y \in U \cap Y$ take a compact open nbhd $V_y$
  properly contained in $Y$.
  This can be done because
  $y$ cannot be $\{1\}$
  as $y$ is a prime filter of the dual of $U$;
  there are $a, a' \in L \setminus \{1\}$ with $a \vee a' = 1$.
  Now 
  apply the compactness of $U \cap Y$ to obtain $V$ the union of a finite subcover of $\{V_y\}$.
  If $L'$ is the dual Heyting algebra of $U$,
  we have $V \in L' \setminus\{0, 1\}$.
  By using Lemma~\ref{lem:homog},
  take $V' \subseteq L'$ such that $V \indep[\0] V'$ and $V \equiv V'$.
  Note that $\tp^{L'}(V'V(V \cap V'))/ \0)=\tp^{L'}(VV'(V\cap V') / \0)$.
  The translation appearing in the proof of Lemma~\ref{lem:homog}
  can be used to show $\tp^{L}(V'V / U)=\tp^{L}(VV' / U)$,
  where $U, V, V'$ are regarded as elements of $L$
  as they are compact open sets.
  By the ultrahomogeneity of $L$,
  there is $\sigma \in \Aut(L)$ such that $\sigma(U) = U$, $\sigma(V) = V'$,
  $\sigma(V') = V$, and $\sigma(V \cap V') = V \cap V'$.
  By construction, $\sigma$ satisfies the conclusion of the lemma.
\end{proof}

For $H \subseteq G$,
let $\Gamma(H)$ be the family of clopen up-sets $Y$ in $X$
such that $\Aut_X(Y) \le H$.

\begin{lem}\label{lem:ideal}
  Let $H \subseteq G$ be such that $|G : H| < \continuum$.
  If $Y, Z \in \Gamma(H)$ are not disjoint,
  then $Y \cup Z \in \Gamma(H)$.
\end{lem}
\begin{proof}
  Let $Y, Z$ be as in the hypothesis and $\sigma \in \Aut_X(Y \cup Z)$.
  Let $W = Y \cap \sigma^\inv(Y) \cap Z$.
  As in \cite[Lemma 2.3]{TRUSS1989494}, one can assume
  that in particular, it contains a nonempty clopen up-set.
  If $Y \subseteq Z$ or $\sigma = \mathrm{id}_X$, then there is nothing to show.
  Thus we assume not.
  This implies the clopen set $Y \setminus W$ is nonempty either.
  Thus we may take a clopen subset $T' \subseteq Y$ containing $Y \setminus W$.
  Applying an argument similar to one in the proof of Lemma~\ref{lem:adam}
  to the dual skeletal interior algebra of $Y$,
  we obtain $\phi \in \Aut_X(Y)$ exchanging $T'$ and $T$ for some
  clopen $T \subseteq W$.
  Note that $\phi$ exchanges $Y \setminus Z$ with some subset of $T$.
  Now consider the inclusion map $Y \hookrightarrow Y \cup Z$.
  This is an Esakia space morphism.
  Its dual $\pi$ is a (surjective) Heyting algebra homomorphism,
  and we have $\pi(T'') = T''$ for all clopen up-sets $T'' \subseteq Y$,
  where $T''$ on the left-hand side and on the right-hand side are an element of $\dom \pi$ and $\ran \pi$, respectively.
  Since $\pi$ is a homomorphism, we see that that the quantifier-free types of
  $T$ and $\sigma(T)$ in $\ran \pi$ are identical
  as those in $\dom \pi$ are already identical
  (note that $\sigma$ can be identified with an automorphism on $\dom \pi$).
  By the ultrahomogeneity of $\ran \pi$,
  and a reasoning via translation as above,
  there is $\theta \in \Aut_X(Y)$
  such that $ \theta \upharpoonright T = \sigma \upharpoonright T$.
  One can argue as in \cite[Lemma 2.3]{TRUSS1989494}
  that $\sigma = \theta\phi(\phi^\inv\theta^\inv\sigma\phi)\phi^\inv \in H$.
\end{proof}

\begin{thm}\label{thm:ssip}
  $G$ has the strong small index property.
\end{thm}
\begin{proof}
  The proof will be an adaptation of the argument in Truss~\cite{TRUSS1989494}.
  Let $H \le G$ be of a countable index.
  We show that $H$ is open by examining $\Gamma := \Gamma(H)$.

  First, observe that if $(Y_n)_{n \in \omega}$ is a family of clopen up-sets
  in $X$
  such that for each family of automorphisms in $\prod_{n <\omega}\Aut_X(Y_n)$
  its agglutinate is in $\Aut(X)$,
  then $Y_n\in\Gamma$ for some $n < \omega$.
  This is proved like \cite[Lemma 2.2]{TRUSS1989494}
  from the hypothesis that $|G : H| < 2^\omega$.

%  This implies that $\Gamma$ is nonempty.
  Let $X^* = \bigcup \Gamma$.
%  (We need to rewrite the following in terms of skeletal interior algebras.)
  If $X \setminus X^*$ contains a clopen up-set $Y_0$,
  then by the universality of $L$,
  one may construct nonempty clopen up-sets $Y_1, Z_1, Y_2, Z_2, \dots$
  such that $Y_n \cap Z_n = \0$, $Y_{n+1}, Z_{n+1} \subseteq Y_n$
  for $n = 1, 2, \dots$.
  Indeed, given a clopen up-set $Y_n$,
  we may apply Lemma~\ref{lem:partition} to $Y_n$ to obtain $Y_{n+1}$ and $Z_n$.
  The family~$(Z_n)_{n < \omega}$ satisfies the hypothesis
  of the claim in the preceding paragraph.
  This contradicts the definition of $X^*$.
  We see that $X^*$ is dense in $X$ in the spectral topology.
  Since $H$ fixes $X^*$ setwise,
  it fixes $X^\dagger := X \setminus X^*$ as well.

  Suppose $X^\dagger \neq \0$ by way of contradiction.
  Suppose further that $X^\dagger$ is finite.
  Then by Lemma~\ref{lem:trans},
  $|G \cdot X^\dagger| = \continuum$.
  Since $X^\dagger$ is fixed setwise by $H$, we obtain the inequality
  \[
    |G : H| \ge |G : G_{\{X^\dagger\}}| = |G \cdot X^\dagger| = \continuum,
  \]
  which contradicts the assumption.
  Hence $X^\dagger$ is infinite (under the assumption that it is nonempty).

  Since $X$ is compact and $X^\dagger$ is closed,
  we may take a basic open $U \subseteq X$ around a cluster point of $X^\dagger$,
  where the topological notions refer to the spectral topology.
  The subset $U$ is a clopen up-set in the patch topology,
  so Lemma~\ref{lem:homog} is applicable.
  Note that $U \cap X^\dagger$ is infinite.
  Let $H$ be the dual of the Esakia space $U$.
  % Enumerate as $\{0, a_i, \neg a_i, 1\}$ ($i<\omega$)
  % the distinct 1-generated free Boolean algebras in $\mathrm{Age}(H)$
  % One can see that for each prime filter $u$ of $H$ there exists $i < \omega$
  % such that exactly one of $a_i$ and $\neg a_i$ belongs to $u$.
  % I.e., $U$ is the disjoint union of the open sets
  % $\bigcup_{i \in I} \widehat{a_i}$ and $\bigcup_{i \in J} \widehat{\neg a_i}$.
  By assumption, at least one (say, $U_0'$) of the sets
  obtained by applying Lemma~\ref{lem:partition}
  contains
  infinitely many elements of $X^\dagger$.
  Let $U_0$ be the other open set.
  Repeat this construction with $U$, $U_0'$, and $U_0$ replaced
  with $U_i$, $U_i'$, and $U_i$, respectively ($i <\omega$).
  Then for every $i < \omega$ there is $j \ge i$ such that $U_j \cap X^\dagger \neq \0$,
  for otherwise $|U'_i| = 1$.
  By taking a subsequence, one may assume that $U_i \cap X^\dagger \neq \0$
  for every $i < \omega$.
  Let $\sigma_i \in \Aut_{X}(U_i)$ be as in the conclusion of Lemma~\ref{lem:adam}
  applied to $U_i$.
  The agglutinate $\sigma_T \in G$ of
  $\sigma_i$ ($i \in T$) and $\mathrm{id}_{U_i}$ ($i \not \in T$)
  exists for
  each $T \subseteq \omega$.
  These are $2^{\aleph_0}$ distinct automorphisms.
  Since $|G : H| \le 2^{\aleph_0}$, there are $T \neq T' \subseteq \omega$
  such that $\sigma_T{\sigma_{T'}}^\inv \in H$.
  By construction, $\sigma_T{\sigma_{T'}}^\inv$ does not preserve $X^\dagger$,
  contrary to the previous observation.
  We have shown that $X^* = X$.
  By arguing the same way as Truss~\cite[Theorem 3.7]{TRUSS1989494},
  each $Y \in \Gamma$ has a unique extension in
  $A := \max \Gamma$, the set of maximal elements of $\Gamma$
  with respect to set inclusion,
  which
  is a pairwise disjoint subfamily
  covering $X$.
  We conclude that $G_{A} = \prod_{Y \in A} \Aut_X(Y)\le H$.

  The \emph{strong} small index property follows from the definability of $A$  from $H$ in the proof above.
  In fact, $A$ is the (finite) set of maximal elements of $\Gamma(H)$,
  so it is fixed by $H$ setwise.
\end{proof}

\begin{cor}\label{cor:wei}
  $L$
  has the DLCF~\cite{casanovas04:_weak}
  (i.e., every definable relations has the least definably closed set of parameters
  that defines it)
  but does not admit the elimination of imaginaries.
\end{cor}
\begin{proof}
  Since $\Aut(L)$ has the strong small index property,
  $L$ has the weak elimination of imaginaries,
  i.e., every definable relations has the least algebraically closed set of
  parameters that defines it.
  Since $\mathrm{Age}(L)$ has the strong amalgamation property,
  definable closures and algebraic closures are the same in $L$.
  On the other hand, the proof of the Theorem
  shows that $L$ does not have codes for finite set of tuples.
  Indeed, let $\{a, b, c\}$ be the atoms of a subalgebra $A_0$ of $L$
  that is isomorphic to the 8-element Boolean algebra.
  Consider $H := G_{\{a,b\}}$.
  Then, $G_{(A)} \le H$
  where $A = \{a, b, c\}$
  (in fact, $\max \Gamma (H) = A$).
  Suppose that there is a finite set $A' \subseteq L$ such that $G_{(A')} = H$.
  Then since $G_{(A)} \le G_{(A')}$, we have $A' \subseteq \dcl(A)$.
  By the strong amalgamation property of the age of $L$,
  the model-theoretic notion of definable closures
  and the universal-algebraic notion of generated subalgebras
  coincide,
  so $\dcl(A) = A_0$.
  Since some element of $H$ switches $a, b$,
  neither $a$ nor  $b$ are $A'$-definable,
  so $a, b, \neg a, \neg b \not \in A'$.
  But for every such $A' \subseteq A_0$, the difference $G_{(A')} \setminus H$
  is nonempty.
\end{proof}

\section{Concluding remarks}
We conclude this article by discussing a newer and more systematic method
for the small index property of ultrahomogeneous structures
and its applicability in the \FR{} limit $L$ of finite Heyting algebras.
The method concerns the \emph{extension property for partial automorphisms} (EPPA) of a \FR{} class $\mathcal K$,
also known as the \emph{Hrushovski property}.
The EPPA  can often be used to prove the existence of \emph{ample generics}
in the automorphism group of the \FR{} limit of $\mathcal K$,
which in turn implies the small index property \cite{kechris06:_turbul_amalg_gener_autom_homog_struc}.
For instance, this is how the small index property was proved for every \emph{free} countable ultrahomogeneous structure \cite{siniora_solecki_2020}
(here a structure is free if each lower amalgam of inclusion maps between elements of its age is completed by the union of their ranges).
However,
this was not the way we proceeded in this article for the following reason.
As in \cite[Proposition~1.5.14]{siniora17:_autom_group_homog_struc}
a countable $\omega$-categorical locally finite ultrahomogeneous
structure contains an infinite chain,
then its age cannot have the EPPA.
Our structure $L$ is not $\omega$-categorical \emph{per se},
but, as we see in the following,
its age does not have the EPPA.
Indeed, there exist $a, b \in L$
such that $a \equiv_\0{} b$ and $a < b$.
For instance, let $\{0 < a < b < 1\}$ be a copy of the 4-chain in $L$,
and check which equations are true of each of $a$ and $b$
by using $\neg a = \neg b = 0$.
Consider then the partial automorphism $\phi: a \mapsto b$ of  $\langle ab \rangle^L$.
If the age of $L$ has the EPPA,
there is a finite structure $B \subseteq L$ and an automorphism $f : B \to B$
extending $\phi$,
and \[a < b = f(a) < f(b) = f(f(a)) < f(f(b)) = f(f(f(a))) < \cdots\] is an infinite chain in $B$.
This is absurd.

It remains an important open question whether $\Aut(L)$ has ample generics
even if not via the EPPA.
\subsection*{Acknowledgment}
The author expresses gratitude to the people at the Czech Academy of Sciences who read this manuscript and provided me with useful suggestions.
In particular, Adam Barto\v{s} pointed out gaps present in my original argument.
The author also thanks the anonymous reviewer for suggesting many improvements on the matter of exposition.
\printbibliography
\end{document}